\documentclass[12pt]{amsart}

\usepackage{amsmath,amssymb,amsthm}
\usepackage{booktabs}
\usepackage{microtype}
\usepackage{url}

\newtheorem{theorem}{Theorem}

\newtheorem{proposition}[theorem]{Proposition}
\newtheorem{corollary}[theorem]{Corollary}

\theoremstyle{definition}
\newtheorem{definition}[theorem]{Definition}
\theoremstyle{remark}
\newtheorem{remark}[theorem]{Remark}
\newtheorem{observation}[theorem]{Observation}

\title{The Collision Transform}

\author{Alexander S.\ Petty}
\email{alexander.petty@gmail.com}
\date{March 2026}
\subjclass[2020]{11A63, 11N05, 11M06}

\begin{document}
\begin{abstract}
The collision invariant $S_{\ell}(p)$, introduced in the
companion paper~\cite{paperA}, is a function of
$p \bmod b^{\ell+1}$ and therefore lives on the finite
group $(\mathbb{Z}/b^{\ell+1}\mathbb{Z})^{\times}$. Its
Fourier expansion over Dirichlet characters modulo
$b^{\ell+1}$ is the \emph{collision transform}. The
reflection identity forces all even-character coefficients
of the centered invariant to vanish: only odd characters
contribute.

The centered prime harmonic sum
$F^{\circ}(s) = \sum_p S^{\circ}_p / p^s$ is therefore a
finite linear combination of non-trivial odd character
sums $\sum_p \chi(p)/p^s$, with no principal-character
term. At $s = 1$, each sum converges by Mertens' theorem
for arithmetic progressions. Convergence below $s = 1$
is conditional on the absence of $L$-function zeros
above a given depth. Computation indicates convergence
persists
to at least $s = 0.6$ in base~$10$ and to $s = 0.5$ in
base~$3$. The real parts of the products
$\hat{S}^{\circ}(\chi) \cdot P(s, \chi)$ have mixed
signs, so convergence is a collective constraint on the
joint zero distribution, not a test of each $L$-function
individually.

Aggregating the collision deviation across bases with
a fixed convergent weighting produces the \emph{base sum},
a function on primes that reveals mod-$3$ structure.
For bases with $3 \nmid b$, the reflection
$a \mapsto m - a$ fixes a unique residue class
modulo~$3$, and the mean of $S$ over units in that
class equals the grand mean $-1/2$ (the neutrality
theorem). Removing the mod-$3$ component introduces a
principal-character term that is absent from
$F^{\circ}$. The base-summed harmonic sum is
negligible: the collision invariant's structural
content is base-specific.
\end{abstract}
\maketitle

\section{The Transform}

By the finite determination theorem~\cite{paperA}, the
collision deviation $S_{\ell}$ is a function on the finite
group $(\mathbb{Z}/m\mathbb{Z})^{\times}$ where
$m = b^{\ell+1}$.

\begin{definition}
The \emph{collision transform} of $S_{\ell}$ is
\[
\hat{S}_{\ell}(\chi) = \frac{1}{\phi(m)}
\sum_{a \in (\mathbb{Z}/m\mathbb{Z})^{\times}}
S_{\ell}(a)\, \overline{\chi}(a)
\]
for each Dirichlet character $\chi$ modulo $m$. The
inverse transform recovers $S_{\ell}$:
\[
S_{\ell}(p) = \sum_{\chi \bmod m}
\hat{S}_{\ell}(\chi)\, \chi(p).
\]
\end{definition}

This is standard Fourier analysis on the finite abelian
group $(\mathbb{Z}/m\mathbb{Z})^{\times}$. The novelty
is the function being analyzed, not the tool.

\section{Structural Properties of the Transform}

The reflection identity from~\cite{paperA} constrains
the transform in a way that has immediate consequences
for which $L$-functions can participate.

\begin{theorem}[Trivial coefficient]\label{thm:trivial}
$\hat{S}_{\ell}(\chi_0) = -1/2$.
\end{theorem}

\begin{proof}
By the reflection identity~\cite{paperA}, the map
$a \mapsto m - a$ pairs units with
$S(a) + S(m{-}a) = -1$. Each pair averages to $-1/2$.
\end{proof}

\begin{definition}
The \emph{centered collision deviation} is
$S^{\circ}(a) = S(a) - \overline{S}_{R(a)}$,
where $\overline{S}_R$ is the average of $S$ over units
in spectral class $R = (a{-}1) \bmod b$.
\end{definition}

\begin{theorem}[Antisymmetry]\label{thm:antisymmetry}
The centered transform satisfies
$\hat{S}^{\circ}(\chi) = 0$ for every even character
($\chi(-1) = 1$).
\end{theorem}

\begin{proof}
The reflection identity $S(a) + S(m{-}a) = -1$ pairs
spectral classes: $R(a) + R(m{-}a) = b - 2$. Averaging
over each paired class gives
$\overline{S}_R + \overline{S}_{b-2-R} = -1$. Since
$S^{\circ}(a) = S(a) - \overline{S}_{R(a)}$:
\[
S^{\circ}(a) + S^{\circ}(m{-}a) =
\bigl[S(a) + S(m{-}a)\bigr]
- \bigl[\overline{S}_R + \overline{S}_{b-2-R}\bigr]
= -1 - (-1) = 0.
\]
For even $\chi$ ($\chi(-1) = 1$),
$\overline{\chi}(m{-}a) = \overline{\chi}(-a)
= \overline{\chi}(a)$, so
\[
\hat{S}^{\circ}(\chi)
= \frac{1}{\phi(m)} \sum_a S^{\circ}(a)\,
\overline{\chi}(a)
= -\hat{S}^{\circ}(\chi),
\]
hence $\hat{S}^{\circ}(\chi) = 0$.
\end{proof}

\begin{corollary}
The centered collision deviation decomposes as
\[
S^{\circ}(p) = \sum_{\substack{\chi \bmod m \\
\chi(-1) = -1}} \hat{S}^{\circ}(\chi)\, \chi(p).
\]
Only odd characters contribute. The bilateral symmetry
of the digit function determines which $L$-functions
the collision transform can see.
\end{corollary}

All identities involving $F^{\circ}$ and
$P(s, \chi) = \sum_p \chi(p)/p^s$ are understood up
to the finite contribution of primes $p \le m$, which
is irrelevant for convergence.

\section{Convergence at $s = 1$}

The natural question is whether the centered collision
deviations cancel across the population of primes.

\begin{theorem}[Centered convergence]\label{thm:convergence}
For every base $b \ge 2$ and lag $\ell \ge 1$,
\[
F^{\circ}(1, \ell) = \sum_{\substack{p > m \\
\gcd(p,b) = 1}} \frac{S^{\circ}_p(\ell)}{p}
\]
converges.
\end{theorem}

\begin{proof}
By the antisymmetry theorem,
\[
F^{\circ}(1) = \sum_{\substack{\chi \bmod m \\
\chi(-1) = -1}} \hat{S}^{\circ}(\chi)
\sum_p \frac{\chi(p)}{p}.
\]
For each non-trivial character $\chi$ modulo $m$, the
prime harmonic sum $\sum_p \chi(p)/p$ converges. This
follows from Mertens' theorem for arithmetic
progressions~\cite{mertens,davenport}: the prime sum in
each residue class modulo $m$ satisfies
\[
\sum_{\substack{p \le x \\ p \equiv a \pmod{m}}}
\frac{1}{p} = \frac{\log\log x}{\phi(m)} + M_a + o(1),
\]
and the character weighting $\sum_a \chi(a) = 0$ cancels
the $\log\log x$ term. Since $F^{\circ}$ is a finite
linear combination of convergent sums, it converges.
\end{proof}

\section{Below $s = 1$}

Convergence at $s = 1$ asks little of each prime: the
weight $1/p$ decays rapidly. Lowering $s$ amplifies
every prime's contribution, demanding that finer
structure project coherently across the population.

For $\operatorname{Re}(s) > 1$, the prime character sum
$P(s, \chi) = \sum_p \chi(p)/p^s$ relates to the
Dirichlet $L$-function~\cite{davenport,iwaniec} by
\[
P(s, \chi) = \log L(s, \chi) - H(s, \chi),
\]
where $H(s, \chi)$ converges for
$\operatorname{Re}(s) > 1/2$.

\begin{theorem}[Conditional penetration]
\label{thm:penetration}
Suppose there exists $\sigma_0 \in (1/2, 1)$ such
that for every odd character $\chi$ modulo~$m$, the
$L$-function $L(s, \chi)$ has no zeros in the
half-plane $\operatorname{Re}(s) > \sigma_0$. Then
$F^{\circ}(s)$ converges for every real $s > \sigma_0$.
\end{theorem}

\begin{proof}
By the antisymmetry theorem,
$F^{\circ}(s) = \sum \hat{S}^{\circ}(\chi)\, P(s, \chi)$
over odd non-trivial characters. It suffices to prove
$P(s, \chi)$ converges for each such $\chi$ and real
$s > \sigma_0$.

First assume $\chi$ is primitive. The explicit formula for
$\psi(x, \chi) = \sum_{n \le x} \Lambda(n)\chi(n)$
gives~\cite{davenport,iwaniec}
\[
\psi(x, \chi) = -\sum_{|\gamma| \le T}
\frac{x^{\rho}}{\rho}
+ O\!\left(\frac{x \log^2(mxT)}{T}\right),
\]
where $\rho = \beta + i\gamma$ runs over non-trivial
zeros of $L(s, \chi)$. By hypothesis $\beta \le \sigma_0$.
Choosing $T = x$ and applying the standard zero-sum
estimate (see~\cite{davenport}, Ch.~17,
or~\cite{iwaniec}, Theorem~5.15):
\[
\psi(x, \chi) \ll x^{\sigma_0} \log^2(mx).
\]
Since $\psi - \vartheta \ll x^{1/2}$ and
$\sigma_0 > 1/2$, we have
$\vartheta(x, \chi) \ll x^{\sigma_0} \log^2(mx)$,
where $\vartheta(x, \chi) = \sum_{p \le x}
\chi(p) \log p$.

Partial summation with
$f(t) = 1/(t^s \log t)$ gives
\[
\sum_{p \le X} \frac{\chi(p)}{p^s}
= \frac{\vartheta(X, \chi)}{X^s \log X}
+ \int_2^X \vartheta(t, \chi)\,
\frac{s \log t + 1}{t^{s+1} (\log t)^2}\, dt.
\]
The boundary term tends to $0$, and the integral
converges because $s > \sigma_0$. Thus $P(s, \chi)$
converges. If $\chi$ is imprimitive, induced from a
primitive $\chi^{*}$, then
$P(s, \chi) - P(s, \chi^{*})$ is a finite sum over
primes dividing $m$, so convergence is unchanged.
Thus $F^{\circ}$ converges as a finite linear
combination.
\end{proof}

\begin{proposition}[Analytic continuation]
\label{prop:analytic}
Define
\[
\mathcal{F}^{\circ}(s) = \sum_{\chi \text{ odd}}
\hat{S}^{\circ}(\chi)\,
\bigl(\log L(s, \chi) - H(s, \chi)\bigr).
\]
Then $\mathcal{F}^{\circ}$ agrees with $F^{\circ}$
for $\operatorname{Re}(s) > 1$ (up to the finite
contribution of primes $p \le m$) and extends
holomorphically to every simply connected region in
$\operatorname{Re}(s) > 1/2$ on which all relevant
odd $L(s, \chi)$ are nonvanishing.
\end{proposition}

\begin{proof}
For $\operatorname{Re}(s) > 1$, the Euler product gives
$P(s, \chi) = \log L(s, \chi) - H(s, \chi)$, so
$\mathcal{F}^{\circ}(s) = F^{\circ}(s)$. On a simply
connected zero-free region $\Omega$ in
$\operatorname{Re}(s) > 1/2$, each $L(s, \chi)$ admits
a holomorphic branch of $\log L(s, \chi)$, and
$H(s, \chi)$ is holomorphic on $\Omega$. The finite
sum $\mathcal{F}^{\circ}$ is therefore holomorphic
on~$\Omega$.
\end{proof}

\begin{table}[h]
\centering
\begin{tabular}{rrrrrr}
\toprule
$s$ & base $3$ & base $7$ & base $10$ & base $12$ \\
\midrule
$1.0$ & $-0.15$ & $-0.14$ & $-0.20$ & $-0.28$ \\
$0.8$ & $-0.19$ & $-0.25$ & $-0.40$ & $-0.62$ \\
$0.6$ & $-0.23$ & $-0.46$ & $-0.84$ & $-1.35$ \\
$0.5$ & $-0.26$ & $-0.64$ & $-1.22$ & $-1.88$ \\
\bottomrule
\end{tabular}
\caption{$F^{\circ}(s)$ across bases at $348{,}488$
primes~\cite{nfield}. Computation indicates convergence
persists
well below $s = 1$.}
\end{table}

\begin{remark}
The products
$\operatorname{Re}(\hat{S}^{\circ}(\chi) \cdot
P(s, \chi))$ have mixed signs: in base~$10$ at
$s = 0.5$, thirteen of twenty are positive and
seven are negative. The collision coefficients
$|\hat{S}^{\circ}(\chi)|$ are negatively correlated
with $|P(s, \chi)|$ (Pearson $\rho \approx -0.25$
across all prime bases tested). The collision
invariant appears to place larger Fourier weight on
characters less sensitive to $L$-function zeros.
These observations suggest that any persistence of
convergence below $s = 1$ would be a collective
phenomenon, constraining the joint zero distribution
of the odd $L$-functions modulo~$m$ rather than each
$L$-function individually.
\end{remark}

\section{The Base Sum}

For a fixed prime $p$, the collision deviation
$S_1(p, b)$ varies across bases. The \emph{base sum}
\[
R(p) = \sum_b w(b)\, S^{\circ}_1(p, b),
\]
with a fixed convergent weighting $w(b)$, aggregates
collision structure across bases. Since $S^{\circ}$
depends on $p \bmod b^2$ at each base $b$, the function
$R(p)$ combines information from the finite groups
$(\mathbb{Z}/b^2\mathbb{Z})^{\times}$ across all bases.

Computation reveals that $R(p)$ has mod-$3$ structure:
at lag~$1$ in every base $b$ with $3 \nmid b$, the
mean of $S$ over units $a \equiv 2 \pmod{3}$ equals
exactly $-1/2$ (the grand mean), while the mean over
$a \equiv 1 \pmod{3}$ deviates. This observation,
stable across all bases tested, motivates the
neutrality theorem in the following section.

\begin{observation}
The base-summed prime harmonic sum
\[
F_R(s) = \sum_p \frac{R(p)}{p^s}
= \sum_b w(b)\, F^{\circ}_b(s)
\]
nearly vanishes. At $s = 0.5$ with
$w(b) = 1/b^2$ and bases $b = 3, \ldots, 31$:
individual $F^{\circ}_b(0.5)$ values range from
$-1.4$ to $+5.1$, but their weighted sum is $-0.01$.
The collision invariant's structural content appears
to be base-specific.
\end{observation}

\section{The Mod-$3$ Structure}

\begin{theorem}[Neutrality]\label{thm:neutral}
Let $3 \nmid b$ and $m = b^{\ell+1}$. Let
$k^{*}$ be the unique solution to
$2k \equiv m \pmod{3}$ (i.e., $k^{*} = 2$ if
$m \equiv 1$ and $k^{*} = 1$ if $m \equiv 2
\pmod{3}$). Then the mean of $S_{\ell}$ over units
$a \equiv k^{*} \pmod{3}$ in
$(\mathbb{Z}/m\mathbb{Z})^{\times}$ equals $-1/2$
(the grand mean).
\end{theorem}

\begin{proof}
The reflection $a \mapsto m - a$ sends class $k$ to
class $m - k \pmod{3}$. It fixes the class $k^{*}$
satisfying $k^{*} \equiv m - k^{*} \pmod{3}$. Since
$3$ is odd, the solution is unique. The reflection
identity $S(a) + S(m{-}a) = -1$~\cite{paperA} pairs
units within this class, and each pair averages
to~$-1/2$.
\end{proof}

\begin{corollary}\label{cor:swap}
The remaining two classes modulo $3$ (other than
$k^{*}$) are swapped by reflection, and their means
satisfy
$\overline{S}_{k} + \overline{S}_{m-k} = -1$.
\end{corollary}

\begin{remark}
At lag $\ell = 1$, $m = b^2 \equiv 1 \pmod{3}$ for
every $b$ with $3 \nmid b$ (since $1^2 \equiv
2^2 \equiv 1$). In this case $k^{*} = 2$: class~$2$
is neutral and class~$1$ carries the bias. Since
primes $p > 3$ satisfy $p \bmod 3 \in \{1, 2\}$,
the prime harmonic sum sees class~$1$ (biased) and
class~$2$ (neutral).
\end{remark}

\begin{theorem}[Perfect cancellation]\label{thm:cancel}
The centered sum $F^{\circ}$ contains no Mertens term.
Its decomposition over characters modulo~$m$ involves
only non-trivial characters:
\[
F^{\circ}(s) = \sum_{\substack{\chi \bmod m \\
\chi(-1) = -1}} \hat{S}^{\circ}(\chi)\, P(s, \chi).
\]
\end{theorem}

\begin{proof}
By the antisymmetry theorem, only odd characters
contribute. Every odd character is non-trivial (the
trivial character is even). Therefore no term
$\sum_p p^{-s}$ appears.
\end{proof}

\begin{corollary}
Let $\mu_3(a)$ be the mean of $S^{\circ}$ over all
units $a' \equiv a \pmod{3}$ in
$(\mathbb{Z}/m\mathbb{Z})^{\times}$. For primes
$p > 3$, write $\mu_3(p) = c_0 + c_1 \chi_3(p)$,
where $\chi_3$ is the non-trivial character
modulo~$3$ and
$c_0 = (\mu_3^{(1)} + \mu_3^{(2)})/2$ is the average
of the two mod-$3$ class means visible to primes. Set
$\mathcal{M}(s) = \sum_p \mu_3(p)/p^s$ and
$F^{\circ\circ}(s) = F^{\circ}(s) - \mathcal{M}(s)$.

When $c_0 \ne 0$, the sum $\mathcal{M}$ contains a
principal-character term $c_0 \sum_p p^{-s}$, and
$F^{\circ\circ}$ carries $-c_0 \sum_p p^{-s}$ to
compensate. These principal-character terms cancel
algebraically in $F^{\circ}$ by
Theorem~\ref{thm:cancel}.
\end{corollary}

\begin{table}[h]
\centering
\begin{tabular}{rrr}
\toprule
$s$ & $F^{\circ}$ & $F^{\circ\circ}$ \\
\midrule
$1.0$ & $0.08$ & $0.77$ \\
$0.7$ & $0.27$ & $12.3$ \\
$0.5$ & $0.46$ & $111$ \\
\bottomrule
\end{tabular}
\caption{Numerically, removing the mod-$3$ component
destroys convergence below $s = 1$. Base~$10$,
$348{,}488$ primes.}
\end{table}

\begin{remark}
In base~$3$ at lag~$1$, $m = 9 \equiv 0 \pmod{3}$,
so the neutrality theorem does not apply (all units
satisfy $a \not\equiv 0$). At lag~$1$ in bases with
$3 \nmid b$, $m = b^2 \equiv 1 \pmod{3}$ and class~$2$
is neutral. In general, the neutral class $k^{*}$
depends on $m \bmod 3$.
\end{remark}

\section{Remarks}

The collision transform converts a digit-dynamic invariant
into a standard object of analytic number theory. The
invariant is new; the tools applied to it are classical.
Convergence at $s = 1$ is a theorem. Convergence below
$s = 1$ is conditional on the absence of $L$-function
zeros and is controlled by the sign structure of the
character decomposition.

The antisymmetry theorem selects which $L$-functions
participate (odd characters only). The neutrality theorem
and perfect cancellation eliminate the principal-character
term. The observed anti-correlation between collision
coefficients and prime character sum magnitudes is
consistent with the persistence of convergence: the
collision invariant's Fourier spectrum appears to be
misaligned with the zero-sensitive directions in
character space.


\end{document}